\documentclass[12pt,oneside,reqno]{amsart}
\usepackage{graphicx}
\usepackage{amssymb}
\usepackage{epstopdf}
\usepackage{amsthm}
\usepackage{amsmath}
\usepackage{latexsym}
\usepackage[usenames,dvipsnames]{color} 
\usepackage{tikz,amsmath,amsfonts}
\usetikzlibrary{decorations.pathreplacing,angles,quotes}
\usepackage{ifthen}
\usepackage{hyperref}
\usepackage{mathtools}
\usepackage{color}
\usepackage[usenames,dvipsnames]{color}

\usepackage[margin=1in]{geometry}

\newcommand{\Z}{\mathbb{Z}}
\newcommand{\SL}{\text{SL}}

\newcommand{\e}{\text{e}}

\renewcommand{\mod}[1]{\text{ (mod } #1)}

\newcommand{\Real}{\text{Re}}
\newcommand{\Imag}{\text{Im}}

\newtheorem{theorem}{Theorem}[section]
\newtheorem{lem}[theorem]{Lemma}

\newtheorem{cor}[theorem]{Corollary}

\newtheorem{prop}[theorem]{Proposition}

\renewcommand{\exp}{\text{exp}}

\begin{document}

\title{Zeros of Newform Eisenstein Series on $\Gamma_0(N)$}
\author{Thomas Brazelton, Victoria Jakicic}

\begin{abstract}
\noindent We examine the zeros of newform Eisenstein series $E_{\chi_1,\chi_2,k}(z)$ of weight $k$ on $\Gamma_0(q_1 q_2)$, where $\chi_1$ and $\chi_2$ are primitive characters modulo $q_1$ and $q_2$, respectively. We determine the location and distribution of a significant fraction of the zeros of these Eisenstein series for $k$ sufficiently large.
\end{abstract}

\maketitle
\section{Introduction}
\subsection{Statement of Results}
The zeros of the classical Eisenstein series $E_k$ for $k\geq 4$ were studied in [RS-D], where it was shown that when $E_k$ is restricted to the standard fundamental domain $\mathcal{F}$, its zeros rest entirely on the boundary $|z|=1$. By contrast, the zeros of weight $k$ Hecke cusp forms equidistribute in $\mathcal{F}$ by [R] and [HS]. For cusp forms as the level tends to infinity, although Quantum Unique Ergodicity is known (see [N], [NPS]), the corresponding equidistribution of zeros is unknown. In this paper we study the zeros of newform Eisenstein series where both the weight and level may vary.\\
	
Let $\chi_1$ and $\chi_2$ be primitive characters modulo $q_1$ and $q_2$, respectively, with $q_1,q_2 > 1$. We consider newform Eisenstein series with nebentypus $\chi_1\overline{\chi_2}$ on the congruence subgroup $\Gamma_0(q_1 q_2)$ of $\SL_2(\Z)$. These Eisenstein series are defined by
\begin{align}
E_{\chi_1,\chi_2,k}(z)=\frac{1}{2}\sum_{(c,d)=1}\frac{\chi_1(c)\chi_2(d)}{(c q_2 z+d)^k},\label{czd-expansion}
\end{align}
where, in order to avoid triviality, we assume that
$$\chi_1(-1)\chi_2(-1)=(-1)^k,\qquad k\geq 3.$$
These series have a Fourier expansion given in [DS, Theorem 4.5.1] as
\begin{align}
E_{\chi_1,\chi_2,k}(z) = e(\chi_1,\chi_2,k)\sum_{n=1}^\infty \Bigg( \sum_{ab=n}\chi_1(a)\overline{\chi_2}(b)b^{k-1} \Bigg) \e( nz ),\label{fourier-expansion}
\end{align}
where $e(\chi_1,\chi_2,k)$ is some constant independent of $z$ whose value does not affect the location of zeros.\\
	
In this paper, we determine the location of a distinguished subset of zeros of $E_{\chi_1,\chi_2,k}(z)$.\\

\begin{figure}
	\begin{center}
	\includegraphics[width=0.5\linewidth]{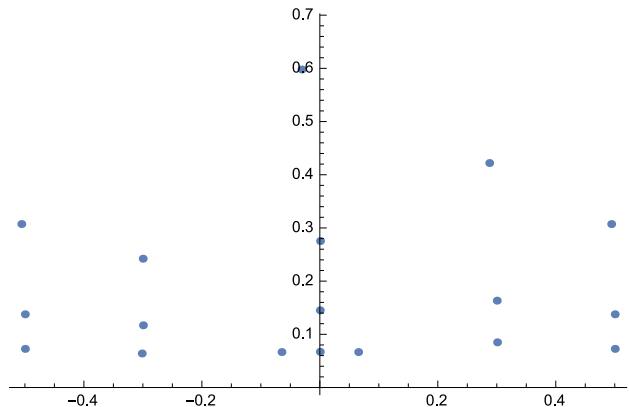}
	\caption{An example for $q_1=3$, $q_2=5$, $k=10$.}	
	\label{fig:czd}
	\end{center}
\end{figure}

We may see a specific example of the vanishing of $E_{\chi_1,\chi_2,k}(x+iy)$ for $-\frac{1}{2}\leq x\leq \frac{1}{2}$ and $\frac{1}{10\sqrt{3}} < y< \frac{\sqrt{10}}{5}$. In Figure \ref{fig:czd}, we have that $\chi_1$ is the Legendre symbol modulo $q_1=3$, $\chi_2$ is the unique character modulo 5 such that $\chi_2(2)=i$, and $k=8$. In this figure and in other similar computations, we notice some approximate vertical lines of zeros, which motivates the following theorem.

\begin{theorem}\label{thm1} Let $a\in\Z$ be such that $\gcd(a,q_2)=\gcd(a+1,q_2)=1$. Then $E_{\chi_1,\chi_2,k}(z)$ has $m$ zeros which are each within $O\left(\frac{1}{q_2 k}\right)$ of the line $x=\frac{a+1/2}{q_2}$. We have that $m$ satisfies
$$m =  \frac{k}{3} - O(\sqrt{k}),$$
with an absolute implied constant.
\end{theorem}

With extra work, one could derive explicit constants for the above theorem. We note that once $k$ is sufficiently large, then one is free to vary $q_1$ and $q_2$ and the results are uniform in these parameters.\\

Additionally, we will demonstrate that the zeros found in Theorem \ref{thm1}, for a fixed integer $a$, are equidistributed with respect to a certain angle $\theta$ defined in \eqref{polar} as $k$ tends to infinity. Furthermore, we see in Section \ref{ineqsection} that if $q_1>3$, these zeros are $\Gamma_0(q_1q_2)$-inequivalent.\\

Theorem \ref{thm1} is approached using the $cz+d$ expansion from \eqref{czd-expansion} for $\frac{1}{2\sqrt{3}q_2} \ll \Imag(z) \ll \frac{\sqrt{k}}{q_2}$.\\

In a complementary range where $\text{Im}(z) \gg \sqrt{k}$, we use the Fourier expansion to approximate $E_{\chi_1,\chi_2,k}(z)$, which is motivated by the ideas of [GS]. Taking the $n=\ell$ and $n=\ell+1$ terms of the Fourier expansion gives a good approximation to $E_{\chi_1,\chi_2,k}(z)$ for $y=\Imag(z)$ in the following range:
\begin{equation*}
    \frac{k-1}{2\pi (\ell + 1)} =: y_{\ell + 1} \leq y \leq y_{\ell} := \frac{k-1}{2 \pi \ell}.
\end{equation*}

\begin{theorem}\label{thm:1.2}
Let $\ell$ be a natural number with $(\ell,q_2)=(\ell+1,q_2)=1$ and $\ell\leq\epsilon\sqrt{k}$ for a sufficiently small $\epsilon>0$ and sufficiently large $k$. Then $E_{\chi_1,\chi_2,k}(z)$ has exactly one zero for $-\frac{1}{2} < x \leq \frac{1}{2}$ and $y_{\ell + 1} \leq y \leq y_{\ell}$.
\end{theorem}

We note that this result is also uniform in $q_1$, $q_2$, and $k$.\\

Due to the constraints of these expansions, we are unable to locate zeros where $\frac{\sqrt{k}}{q_2}\ll y\ll \sqrt{k}$. Note that Theorem \ref{thm1} provides the location of roughly $\gg\varphi(q_2)k$ zeros and Theorem \ref{thm:1.2} provides roughly $\sqrt{k}$ zeros. These zeros are produced in a neighborhood around infinity. In Section \ref{AL}, we study $E_{\chi_1,\chi_2,k}(z)$ near Atkin-Lehner cusps in order to find additional zeros.

\subsection{Heuristic Discussion on Equidistribution}
One of our primary motivations for this work is gathering evidence as to whether the zeros of newform Eisenstein series equidistribute as the level becomes large. A natural way to define equidistribution of a discrete set of points in $\Gamma_0(N)\backslash\mathbb{H}$ is if the points equidistribute in $\Gamma_0(1)\backslash\mathbb{H}$ after application of the projection map
\begin{align*}
\pi:\Gamma_0(N)\backslash\mathbb{H}\to\Gamma_0(1)\backslash\mathbb{H}.
\end{align*}

As an example, Figure \ref{fig:PartialEquid} shows the image of the zeros from Figure \ref{fig:czd} under the map $\pi$.

\begin{figure}[h!]
	\begin{center}
	\includegraphics[width=0.4\linewidth]{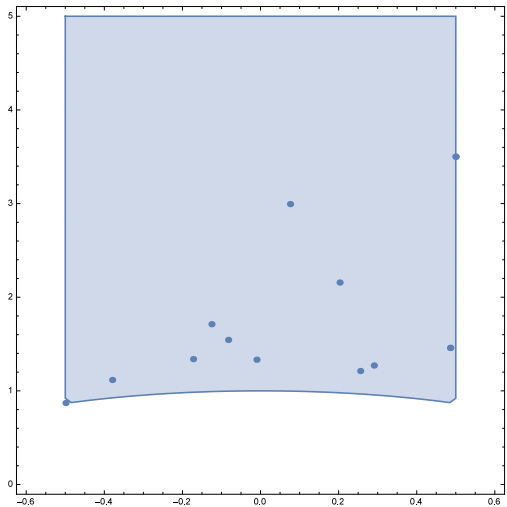}
	\caption{$q_1 = 3$, $q_2 = 7$, $k = 8$}	
	\label{fig:PartialEquid}
	\end{center}
\end{figure}

We note the stark contrast between these zeros, which lie on the interior of $\mathcal{F}$, and the zeros found in [RS-D], which lie entirely on the bottom arc $|z|=1$. Taking $q_2$ to be an odd prime, we see from Theorem \ref{thm1} that we have $q_2-2$ approximate vertical lines of zeros of $E_{\chi_1,\chi_2,k}(z)$ with real part strictly between 0 and 1. Consider the lowest zero from each of these lines, which is within $O\left(\frac{1}{q_2 k}\right)$ of $\omega_a:=\frac{a+1/2}{q_2}+\frac{i}{2\sqrt{3}q_2}$ by Theorem \ref{thm1}. We claim that $\Gamma_0(1)\{\omega_a:1\leq a\leq q_2-2\}$ is a subset of the set of Hecke points, $T_{p}(z)$, for $z=\frac{1}{2}+\frac{i}{2\sqrt{3}}$, and $p=q_2$. For $z\in\mathbb{H}$ and $p$ a prime, the set of Hecke points is defined as the $\Gamma_0(1)$-orbits of the points in the set $\left\{\frac{z+a}{p}:a\mod{p}\right\}\cup\{pz\}$. Note that the set of $\omega_a$'s is missing only three points from $T_p(z)$.\\

It is known that Hecke points $T_p(z)$ equidistribute in $\Gamma_0(1)\backslash\mathbb{H}$ as $p\to\infty$ for $z$ fixed; we refer to [MV] Section 1.2 for discussion on the necessary equidistribution results. If we imagine that our low-lying zeros are modeled by a random perturbation of Hecke points, it is reasonable to believe that they equidistribute in $\Gamma_0(1)\backslash\mathbb{H}$ as $q_2\to\infty$. This discussion gives some heuristic evidence for why the zeros displayed in Figure \ref{fig:PartialEquid} hint at equidistribution.

\section{Outline of the Approach}
In order to determine where the zeros of $E_{\chi_1,\chi_2,k}(z)$ lie, we will first distinguish small regions where the Eisenstein series is well approximated by very few terms. In Section \ref{Tsection}, we will look at two terms from the $cz+d$ expansion, evaluated on a thin vertical strip. In Section \ref{Vsection}, we will look at two terms from the Fourier expansion, which are evaluated in a strip for $y_{\ell + 1} \leq y \leq y_{\ell}$.\\

These regions will be selected so that our main term is sufficiently large along its boundary, and the main term contains one zero within the region. We will then use Rouch\'e's Theorem to demonstrate that the Eisenstein series also has a zero within the designated region. This theorem is restated for the reader's convenience.

\begin{theorem}[\textbf{Rouch\'e}] Let $E$ and $g$ be two complex-valued functions which are holomorphic on a closed region $K$ with rectangular boundary $\partial K$. If the strict inequality holds:
\begin{align*}
|E(z)-g(z)|<|E(z)|+|g(z)|,
\end{align*}
for all  $z\in\partial K$, then $E$ and $g$ have the same number of zeros (including multiplicity) in the interior of $K$.
\end{theorem}

\section{Zeros in the Region $\frac{1}{2\sqrt{3}}\ll q_2\Imag(z)\ll \sqrt{k}$}\label{Tsection}

\subsection{The Main Terms of $E_{\chi_1,\chi_2,k}(z)$ Along the Line $x=\frac{a+1/2}{q_2}$}
We fix an integer $a$ such that $\gcd(a,q_2)=\gcd(a+1,q_2)=1$. In a small region around the line $x=\frac{a+1/2}{q_2}$, we will see that the terms in the $cz+d$ expansion where $c=1$ and $d=-a,-a-1$ are a good approximation to $E_{\chi_1,\chi_2,k}(z)$. We denote these terms as
$$g_a(z):=\frac{\chi_2(-a)}{(q_2 z-a)^k} + \frac{\chi_2(-a-1)}{(q_2 z-a-1)^k}.$$

We first wish to determine where our main term $g_a(z)$ has roots in a small region around the vertical line $x=\frac{a+\frac{1}{2}}{q_2}$. Observe that, in order to have $g_a(z)=0$, we must have that the magnitudes of the two terms are equal. That is,
\begin{align*}
\frac{1}{|q_2 z-a|^k} = \frac{1}{|q_2z-a-1|^k},
\end{align*}
which implies that $x=\frac{a+1/2}{q_2}$. Along this vertical line, we make a substitution to polar coordinates given by
\begin{align}
z=\frac{a+1/2}{q_2}+iy=\frac{a}{q_2}+\frac{R}{q_2}e^{i\theta}\label{polar}.
\end{align}

These coordinates will be used frequently and, for the reader's convenience, are illustrated in Figure \ref{fig:coordinates}. With this substitution, we see that 

\begin{align*}
q_2 z-a &=\frac{1}{2}+iq_2y = Re^{i\theta}\\
q_2 z-a-1&=-\frac{1}{2}+iq_2 y = -Re^{-i\theta}.
\end{align*}

Therefore along $x=\frac{a+1/2}{q_2}$, we may see that $g_a(z)=0$ reduces to
\begin{align*}
0&=\frac{\chi_2(-a)}{(q_2z-a)^k} + \frac{\chi_2(-a-1)}{(q_2z-a-1)^k}=\frac{\chi_2(-a) e^{-i\theta k} + (-1)^k\chi_2(-a-1) e^{i\theta k}}{R^k}.
\end{align*}

This is satisfied if and only if
\begin{align}
e^{2i\theta k} + (-1)^k\chi_2(a)\overline{\chi_2(a+1)} = 0\label{charactersumzero}.
\end{align}
We conclude that the zeros of $g_a(z)$ along the vertical line $x=\frac{a+1/2}{q_2}$ depend only on their angle from the ray emerging from $\frac{a}{q_2}$. It now suffices to show that as $k$ is sufficiently large, both $g_a(z)$ and $E_{\chi_1,\chi_2,k}(z)$ have the same number of zeros in small regions surrounding each zero of $g_a(z)$ along this vertical line.

\subsection{Defining the Regions Containing Zeros}

We fix a positive $\epsilon$ which is sufficiently small, and independent of $q_1$, $q_2$, and $k$. Let $K_{\epsilon}$ denote a small region around our vertical line:
$$K_{\epsilon}=\Bigg\{z=x+iy:\frac{a+\frac{1}{2}-\frac{\epsilon}{k}}{q_2}\leq x\leq \frac{a+\frac{1}{2}+\frac{\epsilon}{k}}{q_2},\ \frac{1}{2\sqrt{3}q_2}<y< \frac{c\sqrt{k}}{q_2} \Bigg\},$$
where $c$ is some small absolute constant. A view of $K_\epsilon$ is displayed as the shaded region in Figure \ref{fig:coordinates}.\\

\begin{figure}[h!]
\begin{center}
\begin{tikzpicture}
\fill[gray!20,nearly transparent] (2.5,0.5)--(3.5,0.5)--(3.5,6)--(2.5,6)--cycle;
\draw (0,0) -- (6,0) -- (6,6) -- (0,6) -- (0,0);
\draw (3,0)--(3,6);
\draw (2.5,0)--(2.5,6);
\draw (3.5,0)--(3.5,6);
\draw (0,0)--(6,2.4);
\draw (0,0)--(6,3.6);
\draw [dashed] (0,1)--(6,1);
\node [below] at (0,0) {$\frac{a}{q_2}$};
\node [below] at (6,0) {$\frac{a+1}{q_2}$};
\node [below] at (3,0) {$\frac{a+\frac{1}{2}}{q_2}$};
\node [above left] at (3.1,1.9) {$z_2$};
\node [below right] at (2.9,1) {$z_1$};
\node [above] at (3.3,6.2) {$\frac{\epsilon}{q_2k}$};
\node [above right] at (1,-0.08) {$\theta_1$};
\node[right] at (6,1.1) {$y=\frac{1}{2\sqrt{3}q_2}+\frac{\eta}{q_2 k}$};
\draw[line width=1.1pt] (2.5,1.2)--(3.5,1.2)--(3.5,1.9)--(2.5,1.9)--cycle;
\draw [fill=white] (3,1.2) circle[radius= 0.2 em];
\draw [fill=white] (3,1.9) circle[radius= 0.2 em];
\draw [decoration={brace,raise=3pt},decorate] (3,6)--(3.5,6);
\node[left] at (2.5,1.8) {$W_1$};
\draw [black,thick,domain=0:22] plot ({cos(\x)}, {sin(\x)});
\filldraw (0,0) circle[radius=1.5pt];
\filldraw (3,1.5) circle[radius=1.5pt];
\draw [dashed] (0,0.5)--(6,0.5);
\node[right] at (6,0.4) {$y=\frac{1}{2\sqrt{3}q_2}$};

\end{tikzpicture}\end{center}
\caption{Coordinates around $x=\frac{a+1/2}{q_2}$}\label{fig:coordinates}
\end{figure}

We will prove that the conditions hold to apply Rouch\'e's Theorem for $g_{a}(z)$ and $E_{\chi_1,\chi_2,k}(z)$ on the boundaries of a series of regions $W_n\subset K_\epsilon$. Each region $W_n$ will be chosen such that it contains exactly one zero of $g_a(z)$, and $|g_a(z)|$ is large on $\partial W_n$.\\

We pick $\theta_1$ to be the smallest positive real number such that the following conditions hold
\begin{align}
\big|e^{2ik\theta_1} +(-1)^k\chi_2(a)\overline{\chi_2(a+1)}\big| =2\label{angle2},
\end{align}
and 
\begin{align}
\tan(\theta_1)>\frac{1}{\sqrt{3}}+\frac{2\eta}{k}\label{tanangle},
\end{align}

for $\eta>0$ large enough (this constant is motivated in the proof of Corollary \ref{samezeros}). Our condition in equation (\ref{tanangle}) is equivalent to stating that the point at $z_1:=\frac{a}{q_2}+r_1e^{i\theta_1}$ with $\Real(z)=\frac{a+1/2}{q_2}$ satisfies $\Imag(z)>\frac{1}{2\sqrt{3}q_2}+\frac{\eta}{q_2 k}$. This guarantees that it lies sufficiently far from the lower boundary of $K_\epsilon$.\\


Note also that if $\theta_1$ satisfies \eqref{angle2}, then $\theta_{n+1}:=\theta_1+\frac{n\pi}{k}$ also satisfies \eqref{angle2} for each natural number $n$. With this in mind, we define $z_n$ for $n\geq 1$ to satisfy

\begin{align*}
z_n = \frac{a}{q_2}+r_n e^{i(\theta_1+\frac{n\pi}{k})},\qquad \Real(z_n)=\frac{a+1/2}{q_2}.
\end{align*}

One may additionally verify that
\begin{align}
|g_a(z_n)| = \frac{2}{r_n^k},\label{main g_a val}
\end{align}
where $r_n=\sqrt{\frac{1}{4}+q_2^2\cdot\Imag(z_n)^2}$.\\

Note that $z_1$ and $z_2$ are illustrated in Figure \ref{fig:coordinates} as the two open points. We finally define $W_n$ to be the rectangle in $K_{\epsilon}$ given by
$$W_n=\{z=x+iy\in K_{\epsilon}: \Imag(z_n)\leq y\leq \Imag(z_{n+1})\}.$$

Additionally, by our definition of $\theta_1$, we have that $W_1$ is the lowest-lying region of this form to lie within $\frac{\eta}{k}$ inside of $K_\epsilon$. Its boundary is illustrated as the bolded rectangle in Figure \ref{fig:coordinates}.\\

We also see that in Figure \ref{fig:coordinates}, the black dot inside of $W_1$ is the unique zero of $g_a(z)$ inside of $W_1$. In general, we may define $\theta_n'=\frac{\theta_n+\theta_{n+1}}{2}=\theta_n+\frac{\pi}{2k}$, and see that this is the unique argument in equation \eqref{charactersumzero} which yields a zero of $g_a(z)$ inside $W_n$. This discussion is summarized in the following Proposition:

\begin{prop}\label{maintermzero} The main term $g_a(z)$ has exactly one zero in each region $W_n$ occurring at the argument $\theta_n'$ in the coordinates defined in equation \eqref{polar}.
\end{prop}

We now determine the number of zeros of $g_a(z)$ in $K_\epsilon$.

\subsection{Bounds on the Number of Zeros of $g_a(z)$ in $K_\epsilon$}

It suffices to determine the highest $W_n$ which is contained within $K_\epsilon$. We define $m$ to be the largest integer satisfying
\begin{align*}
\frac{1}{2q_2}\tan\left(\theta_1+\frac{m\pi}{k}\right)<\frac{c\sqrt{k}}{q_2}.
\end{align*}
for some sufficiently small constant $c$. That is, the imaginary part of the point $z_m$ is less than our upper height on $K_\epsilon$. Note that $\theta_1+\frac{m}{k}\pi \leq \arctan\left(2c\sqrt{k}\right)$, so
\begin{align*}
m &=  \frac{k}{\pi}\left(\arctan\left(2c\sqrt{k}\right)-\arctan\left(\frac{1}{\sqrt{3}}+\frac{2 \eta}{k}\right)\right) +O(1)\\
&=\frac{k}{3}-O(\sqrt{k}).
\end{align*}

This will lead us to prove that the conditions for Rouch\'e's Theorem hold for each $W_n$ where $1\leq n\leq m-1$.

\subsection{Inequalities on $\partial W_n$}\label{inequalities Wn}
Before stating our lemmas, we first include a brief proposition which will be useful for deriving lower bounds for $g_a(z)$.
\begin{prop}\label{epsconstant}
For a complex number $Y$ with $|Y|>1$ and $\delta$ complex such that $|\delta|$ is sufficiently small, we have that
\begin{align*}
\Big|Y+\frac{\delta}{k}\Big|^k=\Big|Y\Big|^k \left(1+O\left(\left|\frac{\delta}{Y}\right|\right)\right).
\end{align*}
\end{prop}
\begin{proof}
\begin{align*}
\left| Y+\frac{\delta}{k}\right|^k &= \exp\left(k\log\left(\left|Y+\frac{\delta}{k}\right|\right)\right)=\exp\left(k\log(\left|Y\right|)+kO\left(\frac{|\delta|}{|Y|k}\right)\right)\\
&=\left|Y\right|^k\cdot\exp\left(O\left(\left|\frac{\delta}{Y}\right|\right)\right)=\Big|Y\Big|^k \left(1+O\left(\left|\frac{\delta}{Y}\right|\right)\right).\qedhere
\end{align*}
\end{proof}

We will now prove the following Lemmas.

\begin{lem}\label{RVYbound} In $K_{\epsilon}$, we have that, with an absolute implied constant,
\begin{align}
\Big|E_{\chi_1,\chi_2,k}(z)-g_a(z)\Big| \ll (2q_2 y)^{-k} + q_2 y \left(\frac{9}{4} + q_2^2 y^2\right)^{-k/2}.\label{upperE-g}
\end{align}
\end{lem}

\begin{lem}\label{g_ell on W}
The following bound holds for $g_a(z)$ on the boundary $\partial W_n$, where $1\leq n\leq m-1$:
\begin{align}
|g_a(z)|\gg \left(\frac{1}{4}+q_2^2 y^2\right)^{-k/2-1}.\label{lowerg_a}
\end{align}
\end{lem}
Lemma \ref{RVYbound} follows directly from the proofs of Lemmas 3.2 and 3.4 in [RVY] by bounding similar terms in the Eisenstein series and omitting any normalizing constants. We will prove Lemma \ref{g_ell on W} below, but first we note that by these lemmas, we obtain the following corollary as a result:

\begin{cor}\label{samezeros} 
For sufficiently large $k$, we have that $E_{\chi_1,\chi_2,k}(z)$ has a unique zero in $W_n$ for each $1\leq n\leq m-1$.
\end{cor}
\begin{proof}[Proof of Corollary \ref{samezeros}]
On the boundary $\partial W_n$, where $1\leq n\leq m-1$, we claim that the following inequality holds:
$$|g_a(z)|>|E_{\chi_1,\chi_2,k}(z)-g_a(z)|.$$
Using Lemmas \ref{RVYbound} and \ref{g_ell on W} on $\partial W_n$, we now argue that the bounds in \eqref{lowerg_a} are significantly larger than those in \eqref{upperE-g}, so long as $q_2 y$ is within a determined range. We have that
\begin{align*}
q_2 y\left(\frac{9}{4} + q_2^2 y^2\right)^{-k/2}=o\left(\left(\frac{1}{4}+q_2^2 y^2\right)^{-k/2-1}\right),
\end{align*}
provided that
\begin{align*}
q_2 y = o(\sqrt{k}).
\end{align*}
Additionally, we have that
\begin{align*}
(2q_2 y)^{-k}=o\left(\left(\frac{1}{4}+q_2^2 y^2\right)^{-k/2-1}\right),
\end{align*}
as $k$ tends to infinity. We may also see that

\begin{align*}
\frac{(2q_2y)^{-k}}{\left(\frac{1}{4}+q_2^2y^2\right)^{-k/2-1}}
\end{align*}

becomes sufficiently small when
\begin{align*}
q_2 y>\frac{1}{2\sqrt{3}}+\frac{\eta}{k},
\end{align*}
for an absolute constant $\eta$ which is large enough compared to the implied constants in \eqref{upperE-g} and \eqref{lowerg_a}.\\

On $\partial W_n$, this gives us that $|E_{\chi_1,\chi_2,k}-g_a(z)|$ vanishes quicker than $|g_a(z)|$ as $k$ tends to infinity. In particular, for a sufficiently large $k$, we have that
\begin{align*}
|g_a(z)| >\big|E_{\chi_1,\chi_2,k}(z)-g_a(z)\big|.
\end{align*}
Then by Rouch\'e's Theorem, we obtain that $g_a(z)$ and $E_{\chi_1,\chi_2,k}(z)$ have the same number of zeros in $W_n$ for each $1\leq n\leq m-1$. By Proposition \ref{maintermzero}, the result follows.
\end{proof}

We note that Theorem \ref{thm1} follows from the previous lemmas and corollary. It only remains to prove Lemma \ref{g_ell on W}.

\begin{proof}[Proof of Lemma \ref{g_ell on W}]
First we look at the right boundary of $K_\epsilon$ , where $\Real(z)=\frac{a+\frac{1}{2}+\frac{\epsilon}{k}}{q_2}$. By the reverse triangle inequality on this vertical line segment,
\begin{align*}
|g_a(z)| \geq \frac{1}{ \left( \left(\frac{1}{2}+\frac{\epsilon}{k}\right)^2 + q_2^2 y^2\right) ^{k/2}}\left(\left(\frac{\left(\frac{1}{2}+\frac{\epsilon}{k}\right)^2 + q_2^2 y^2}{{\left(\frac{1}{2}-\frac{\epsilon}{k}\right)^2 + q_2^2 y^2}}\right)^{k/2} -1\right).
\end{align*}
Expanding the term on the right, we obtain

\begin{align*}
\left(\frac{\left(\frac{1}{2}+\frac{\epsilon}{k}\right)^2 + q_2^2 y^2}{{\left(\frac{1}{2}-\frac{\epsilon}{k}\right)^2 + q_2^2 y^2}}\right)^{k/2} -1&=\left(1+\frac{2\frac{\epsilon}{k}}{\left(\frac{1}{2}-\frac{\epsilon}{k}\right)^2+q_2^2 y^2}\right)^{k/2}-1\\
&\geq \exp\left(\frac{k}{2}\log\left(1+\frac{2\epsilon}{k\left(\frac{1}{4}+q_2^2y^2\right)}\right)\right)-1\\
&=\exp\left(\frac{\epsilon}{\left(\frac{1}{4}+q_2^2y^2\right)}+O\left(\frac{\epsilon^2}{k\left(\frac{1}{4}+q_2^2 y^2\right)^2}\right)\right)-1\gg\frac{\epsilon}{\left(\frac{1}{4}+q_2^2 y^2\right)}.
\end{align*}

This gives us that on the right boundary of $K_\epsilon$,
\begin{align*}
|g_a(z)|\gg_\epsilon \frac{1}{|\frac{1}{4}+q_2^2 y^2|^{k/2+1}},
\end{align*}
by Proposition \ref{epsconstant}.\\

A symmetric argument holds when $\Real(z)=\frac{a+\frac{1}{2}-\frac{\epsilon}{k}}{q_2}$. We now turn our attention to the bottom segment of $\partial W_n$.\\

Recall from equation \eqref{main g_a val}, we have that
\begin{align*}
|g_a(z_n)|=\frac{2}{\left(\frac{1}{4}+q_2^2 y^2\right)^{k/2}}.
\end{align*}
Letting $\delta$ vary from $0\leq\delta\leq\epsilon$, we have $z=z_n+\frac{\delta}{q_2 k}$ on the lower boundary of $\partial W_n$. We may use Proposition \ref{epsconstant} to show

\begin{align*}
\frac{\chi_2(-a)}{(\frac{1}{2} + \frac{\delta}{k}+ iq_2 y)^k} = \frac{\chi_2(-a)}{(\frac{1}{2} + iq_2 y)^k\left(1+\frac{\delta}{k\left(\frac{1}{2}+iq_2 y\right)}\right)^k} =\frac{\chi_2(-a)}{(\frac{1}{2} + iq_2 y)^k}\left(1+O\left(\frac{\epsilon}{\left|\frac{1}{2}+iq_2 y\right|}\right)\right).\\
\end{align*}

We use this to see that
\begin{align*}
|g_a(z)|=|g_a(z_n)|\left(1+O\left(\frac{\epsilon}{\left|\frac{1}{2}+iq_2 y\right|}\right)\right)\gg_\epsilon \left|\frac{1}{4}+q_2^2 y^2\right|^{-k/2}.
\end{align*}

\end{proof}

\subsection{On the $\Gamma_0(q_1 q_2)$-inequivalence of Zeros}\label{ineqsection}

\begin{prop}\label{ineqzeros} The points contained in the region with $-\frac{1}{2}<x
\leq\frac{1}{2}$ and $y>\frac{1}{q_1q_2}$ are $\Gamma_0(q_1 q_2)$-inequivalent.\end{prop}
\proof Suppose we have two points $z$ and $z'$ in $K_\epsilon$ such that $\Imag(z')\geq\Imag(z)$, and a $\gamma\in\Gamma_0(q_1 q_2)$ such that $\gamma z=z'$. Then we must have that
\begin{align*}
|cz+d|^2\leq 1,
\end{align*}
if $\gamma=\left(\begin{smallmatrix} a & b\\ c& d\end{smallmatrix}\right)$. We then have that $c^2 y^2\leq 1$, which implies that $|c|< q_1 q_2$ by our lower bound on $y=\Imag(z)$ in $K_\epsilon$. Since $c\equiv0\mod{q_1 q_2}$ we must have that $c=0$ and $d=\pm 1$, that is, $\gamma$ is a translate.\qed \\
\begin{cor} If $q_1>3$, all the zeros found in Theorem \ref{thm1} are $\Gamma_0(q_1 q_2)$-inequivalent.

\end{cor}

\subsection{A Remark On $(a+1,q_2)>1$}

In the case where $a$ and $a+1$ are not both coprime to $q_2$, we may carry out a similar argument. Let $(a,q_2)=(a+b,q_2)=1$ such that $(a+t,q_2)>1$ for all integers $0<t<b$. We then obtain zeros approaching the line $x=\frac{a+b/2}{q_2}$ given by 
$$z=\frac{a+b/2}{q_2}+iy =\frac{a}{q_2}+Re^{i\theta},$$
for $\theta$ satisfying
\begin{align*}
\big| e^{2i\theta k} +(-1)^k \chi_2(a)\overline{\chi_2(a+b)}\big|=0.
\end{align*}
We can illustrate this in the special case of finding zeros around the imaginary axis in $\mathbb{H}$.\\

In a neighborhood around $x=0$, we have that the main terms of $E_{\chi_1,\chi_2,k}(x+iy)$ are
\begin{align*}
g(z)=\frac{\chi_2(-1)}{(q_2 z -1)^k}+\frac{\chi_2(1)}{(q_2 z +1)^k}.
\end{align*}
We have the same upper bounds on $\big|E_{\chi_1,\chi_2,k}(z)-g(z)\big|$ from the above theorem and analogous lower bounds for $|g(z)|$ on the boundary. We then obtain zeros approaching the line $x=0$ as $k$ tends to infinity, when
\begin{align*}
\big|e^{2i\theta k} + (-1)^k\chi_2(-1)\big| =0.
\end{align*}

However we note that $(-1)^k\chi_2(-1)=\chi_1(-1)$. This gives us that the zeros of the main term are of the form
$$\frac{-1}{q_2}+Re^{i\theta}=iy.$$
Therefore these zeros depend only on the parameters $\text{sgn}(\chi_1)$ and $q_2$.

\color{black}
\section{Zeros in the Region $\Imag(z) \gg \sqrt{k}$}\label{Vsection}

\subsection{The Fourier Expansion}

For this portion of the paper, let
\begin{equation*}
    F(z) = \sum_{n=1}^{\infty} \left[ \sum_{ab = n} \chi_1(a) \overline{\chi_2}(b) b^{k-1} \right] \e( nz ).
\end{equation*}
Recalling \eqref{fourier-expansion}, the zeros of $F(z)$ are the zeros of $E_{\chi_1,\chi_2,k}(z)$. In the definition of $F(z)$, we let
\begin{equation}\label{eq:fsubj}
f_n(z) = \overline{\chi_2}(n) n^{k-1} \e( nz ),
\end{equation}
and take $a = 1$ and $b = n$ to simplify the expansion to
\begin{equation*}
    F(z) = \sum_{n=1}^{\infty} f_n(z) + \delta(z),
\end{equation*}
where
\begin{equation}\label{eq:delta}
    \left| \delta(x+iy) \right| \leq \sum_{n=1}^{\infty} n^{k-1} \exp(-2\pi ny) \Bigg( \sum_{\substack{b \mid n \\ b < n}} \left( \frac{b}{n} \right)^{k-1} \Bigg).
\end{equation}

\subsection{The Main Term of $F(z)$}\label{sec:MainTerm}

Now, we define our main term, denoted $h_{\ell}(z)$. For our purposes, assume $\ell \leq \epsilon\sqrt{k}$ for a sufficiently small $\epsilon > 0$, and assume also that $\ell$ and $\ell + 1$ are coprime to $q_2$. Then, we consider two terms of the Fourier expansion, $n = \ell$ and $n = \ell + 1$, and we define $h_{\ell}(z)$ to be the main term of the expansion:
\begin{equation*}
    h_{\ell}(z) = f_{\ell}(z) + f_{\ell + 1}(z),
\end{equation*}
where we take $z$ to be restricted to the region:
\begin{equation}\label{eq:YBounds}
    \frac{k-1}{2\pi (\ell + 1)} =: y_{\ell + 1} \leq y \leq y_{\ell} := \frac{k-1}{2 \pi \ell}.
\end{equation}

Write
\begin{equation}\label{eq:NormFwithErr}
    F(z) = h_{\ell}(z) + \beta(z),
\end{equation}
where $\beta(z) = f_{\ell + 2}(z) + f_{\ell + 3}(z) + f_{\ell - 1}(z) + f_{\ell - 2}(z) + \varepsilon_1(z) + \varepsilon_2(z) + \delta(z)$ and where $\varepsilon_1(z) = \sum_{n = 1}^{\ell - 3} \; f_n(z)$ and $\varepsilon_2(z) = \sum_{n = \ell + 4}^{\infty} \; f_n(z)$.\\

Now we may find zeros of $h_{\ell}(z)$ in the region from \eqref{eq:YBounds}.

\begin{lem}\label{lem:ZeroH}
The main term $h_{\ell}(z)$ has a unique zero $x_0 + i y_0$ in the region $-\frac{1}{2} < x \leq \frac{1}{2}$ and $y_{\ell + 1} \leq y \leq y_{\ell}$, with $x_0$ and $y_0$ given by \eqref{eq:xnot} and \eqref{eq:ynot} below. 
\end{lem}

\begin{proof}[Proof of Lemma \ref{lem:ZeroH}]
Setting $h_{\ell}(z) = 0$, we find that
\begin{equation*}
    -\overline{\chi_2}(\ell) \chi_2(\ell + 1) \left( 1 - \frac{1}{\ell + 1} \right)^{k-1} = \e(x) \exp(-2 \pi y).
\end{equation*}
Then,
\begin{equation*}
    -\overline{\chi_2}(\ell) \chi_2(\ell + 1) = \e( x ) \quad \text{and} \quad \left( 1 - \frac{1}{\ell + 1} \right)^{k-1} = \exp(-2 \pi y).
\end{equation*}
Consequently, $x_0 \in \left( -\frac{1}{2}, \frac{1}{2} \right]$ is the unique solution to
\begin{equation}\label{eq:xnot}
    \e ( x_0 ) = -\overline{\chi_2}(\ell) \chi_2(\ell + 1)
\end{equation}
and
\begin{equation}\label{eq:ynot}
    y_0 = -\frac{(k-1)}{2 \pi} \log\left(1 - \frac{1}{\ell + 1}\right) = \frac{(k-1)}{2 \pi}\left| \log\left( 1 - \frac{1}{\ell + 1} \right) \right|.
\end{equation}
Using
\begin{equation*}
    -\frac{1}{\ell} \leq \log\left( 1 - \frac{1}{\ell + 1} \right) \leq - \frac{1}{\ell + 1},
\end{equation*}
we see that $y_0 \in (y_{\ell + 1}, y_{\ell})$, as desired.
\end{proof}

\subsection{Method to Prove Theorem \ref{thm:Fourier}}

We define 
\begin{equation*}
    N(y,k) = \frac{(2 \pi y)^k}{\Gamma(k)}
\end{equation*}
to be a natural normalization factor of $F(z)$. Make note that multiplying $F(z)$ by $N(y,k)$ will not affect the zeros of $F(z)$. Now, we explain the use of Rouch\'e's Theorem in the context of this section. We must show that the strict inequality
\begin{equation}\label{eq:FormalIneq}
    N(y,k)\left| \beta(z) \right| = N(y,k)\left| F(z) - h_{\ell}(z) \right| < N(y,k)\left| F(z) \right| + N(y,k) \left| h_{\ell}(z) \right|
\end{equation}
holds in the region $V_{\ell}$, where 
\begin{equation*}
    V_{\ell} = \left\{ z = x + iy : x_0 - \frac{1}{2} \leq x \leq x_0 + \frac{1}{2}, \; y_{\ell + 1} \leq y \leq y_{\ell} \right\}.
\end{equation*}
If \eqref{eq:FormalIneq} holds, then $F(z)$ will have the same number of zeros as $h_{\ell}(z)$ in the region $V_{\ell}$. On $\partial V_{\ell}$, we will show
\begin{equation}\label{eq:InformalIneq}
    N(y,k)\left| \beta(z) \right| < N(y,k)\left| h_{\ell}(z) \right|,
\end{equation}

which implies \eqref{eq:FormalIneq}. This leads us to our main theorem, which is the same as Theorem \ref{thm:1.2}.

\begin{theorem}\label{thm:Fourier}
The function $E_{\chi_1,\chi_2,k}(z)$ has exactly one zero in the region $V_{\ell}$.
\end{theorem}

To prove Theorem \ref{thm:Fourier}, we need the following:

\begin{lem}\label{lem:Functions}
Let $0 < t \leq \frac{1}{2}$, and define
\begin{align}
    M(t) &= \log(t) - \log(| \log(1-t) |) - \frac{\log(1-t) + t}{t} \label{eq:MT} \\ 
    R(t) &= t - \log(1 + t). \label{eq:RT}
\end{align}
Then,
\begin{equation*}
    0 < M(t) < R(t) < t^2.
\end{equation*}
\end{lem}

The proof of Lemma \ref{lem:Functions} is deferred until later in this section. We include the graph of $M(t)$ and $R(t)$ in Figure \ref{fig:Functions} for the reader's convenience.\\

\begin{figure}[h!]
	\begin{center}
	\includegraphics[width=0.6\linewidth]{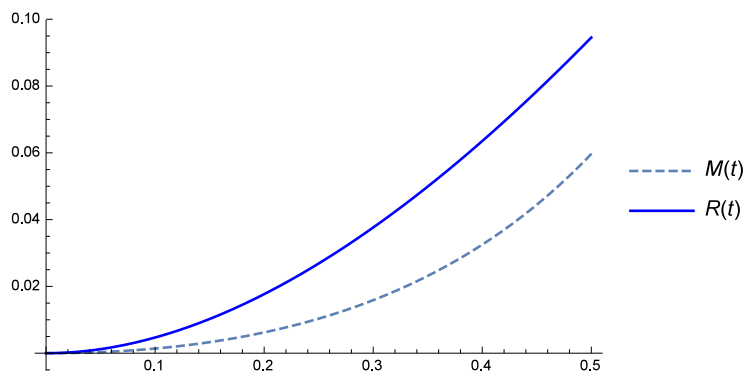}
	\caption{The Functions $M(t)$ and $R(t)$}	
	\label{fig:Functions}
	\end{center}
\end{figure}

Applying Lemma \ref{lem:Functions}, we will show:

\begin{lem}\label{lem:Main}
On $\partial V_{\ell}$,
\begin{equation}\label{eq:BoundonH}
    N(y,k)\left| h_{\ell}(z) \right| \gg \frac{\sqrt{k}}{\ell} \emph{exp}\left( -k \cdot M\left(\frac{1}{\ell + 1}\right) \right).
\end{equation}
\end{lem}

\begin{lem}\label{lem:Error}
For all $z \in V_{\ell}$,
\begin{equation}\label{eq:BoundonErr}
    N(y,k)\left| \beta(z) \right| \ll \frac{\sqrt{k}}{2^k \ell} + \frac{\sqrt{k}}{\ell} \emph{exp}\left( -k \cdot R\left( \frac{1}{\ell + 1} \right) \right).
\end{equation}
\end{lem}

Theorem \ref{thm:Fourier} will follow from Lemmas \ref{lem:ZeroH}, \ref{lem:Functions}, \ref{lem:Main}, and \ref{lem:Error}.

\subsection{Proofs of Lemmas \ref{lem:Main} and \ref{lem:Error}}

Before beginning the proofs of Lemma \ref{lem:Main} and \ref{lem:Error}, we must provide two facts that will aid us. If $0< u < 1$, then
\begin{equation}\label{eq:Fact1}
    u - \frac{u^2}{2} \leq \log\left(1 + u\right) \leq u - \frac{u^2}{4}.
\end{equation}
Additionally, recall Stirling's approximation: For $n \in \mathbb{Z}^+$,
\begin{equation}\label{eq:Stirling}
    \Gamma(n) \sim \sqrt{\frac{2 \pi}{n}} \Bigg( \frac{n}{e} \Bigg)^n.
\end{equation}

\begin{proof}[Proof of Lemma \ref{lem:Main}]
Consider two cases:
\vskip 0.1 in
\textbf{Top and Bottom Boundaries:} Let $y = y_{\ell}$. Note that, from the definition of $h_{\ell}(z)$ and the triangle inequality,
\begin{equation*}
    N(y,k)\left| h_{\ell}(z) \right| \geq \frac{(2 \pi y \ell)^k}{\Gamma(k)} \frac{ \exp(-2\pi \ell y)}{\ell} \cdot \Bigg| 1 - \bigg| \frac{f_{\ell + 1}(z)}{f_{\ell}(z)} \bigg| \Bigg|.
\end{equation*}
In the next steps, we substitute $k - 1 = 2\pi y \ell$ and apply \eqref{eq:Stirling}, which implies
\begin{equation*}
    \frac{e}{\ell} \left(\frac{k-1}{e}\right)^k \frac{1}{\Gamma(k)} \gg \frac{\sqrt{k}}{\ell}.
\end{equation*}
Furthermore, we get an upper bound on $\left| \frac{f_{\ell + 1}(z)}{f_{\ell}(z)} \right|$ at $y = y_{\ell}$ using \eqref{eq:Fact1}, namely
\begin{equation}\label{eq:ExpBound}
    \left| \frac{f_{\ell + 1}(z)}{f_{\ell}(z)} \right| = \left( 1 + \frac{1}{\ell} \right)^{k-1}\exp\left(-\frac{k-1}{\ell} \right) \ll \exp\left(- \frac{k}{4\ell^2} \right).
\end{equation}
Recall that $\ell \leq \epsilon\sqrt{k}$ for a sufficiently small $\epsilon > 0$, so \eqref{eq:ExpBound} is less than $\frac{1}{2}$ (say). Then,
\begin{equation}\label{eq:BoundonHforY}
    N(y,k)|h_{\ell}(z)| \gg \frac{\sqrt{k}}{\ell},
\end{equation}
for $y = y_{\ell}$. Letting $y = y_{\ell + 1}$, by similar methods, we conclude that \eqref{eq:BoundonHforY} holds.

\vskip 0.1 in
\textbf{Left and Right Boundaries:} Let $x = x_0 \pm \frac{1}{2}$. Then, let $r_1(y)$ and $r_2(y)$ be the magnitudes of the two terms in $N(y,k)h_{\ell}(z)$:
\begin{equation*}
    \begin{aligned}
    r_1(y) &= y^k \ell^{k-1} \exp(-2\pi \ell y) \\
    r_2(y) &= y^k(\ell + 1)^{k-1} \exp(-2 \pi (\ell + 1)y).
    \end{aligned}
\end{equation*}
Note that, when $x = x_0 \pm \frac{1}{2}$, we have that
\begin{align*}
N(y,k)\left| h_{\ell}(z) \right| = \frac{(2\pi)^k}{\Gamma(k)}\left(r_1(y) + r_2(y)\right).
\end{align*}
We additionally see that $r_1(y) + r_2(y) \geq \text{max}(r_1(y),r_2(y))$. By elementary calculus, we note that $r_1$ is strictly increasing on $y_{\ell + 1} \leq y \leq y_{\ell}$. Furthermore, $r_2$ is strictly decreasing for the same range of $y$ values. Thus, $\text{max}(r_1(y),r_2(y))$ is minimized at $y_0$, and $r_1 = r_2$ at $y_0$. Hence, if $x = x_0 \pm \frac{1}{2}$, recall \eqref{eq:ynot} and Lemma \ref{lem:Functions} to gain the following:
\begin{align*}
    N(y,k)\left| h_{\ell}(z) \right| &\gg \frac{(2\pi y_0)^k}{\Gamma(k)} (\ell + 1)^{k-1} \exp(-2\pi (\ell + 1)y_0) \\
    &= \frac{ (k-1)^k (\ell + 1)^k}{\Gamma(k)(\ell + 1)} \left|\log\left(1 - \frac{1}{\ell + 1}\right)\right|^k \exp (-(k-1)) \; \cdot \\
    & \qquad \qquad \qquad \exp\left( (k-1)(\ell + 1) \left( \log\left( 1 - \frac{1}{\ell + 1} \right) + \frac{1}{\ell + 1} \right) \right) \\
    &= \frac{e}{\ell + 1} \left( \frac{k-1}{e} \right)^k \frac{1}{\Gamma(k)}\exp\left( -k \cdot M\left( \frac{1}{\ell + 1} \right)-(\ell+1)\left(\log\left(1-\frac{1}{\ell+1}\right)+\frac{1}{\ell+1}\right)\right)\\
    &\gg\left( \frac{k-1}{e} \right)^k \frac{1}{\Gamma(k)}\exp\left( -k \cdot M\left( \frac{1}{\ell + 1} \right)\right).
\end{align*}
Using \eqref{eq:Stirling}, we derive \eqref{eq:BoundonH} in this case.\qedhere

\end{proof}

It will be helpful to note that a minor modification of the proof of Lemma \ref{lem:Main} gives
\begin{equation}\label{eq:UpperonH}
N(y,k) \left| h_{\ell}(z) \right| \ll \frac{\sqrt{k}}{\ell},
\end{equation}
in $V_{\ell}$ as well.

\begin{proof}[Proof of Lemma \ref{lem:Error}]
To begin, we note that $\left| \beta(z) \right| \leq \left| f_{\ell + 2}(z) \right| + \left| f_{\ell - 1}(z) \right| + \left| \varepsilon_1(z) \right| + \left| \varepsilon_2(z) \right| + \left| \delta(z) \right|$. Now, we break the proof into three parts:
\vskip 0.1 in
\textbf{Part 1:} Consider $f_{\ell+2}(z)$ and $f_{\ell - 1}(z)$. In the region $V_{\ell}$, $N(y,k)\left|f_{\ell + 2}(z)\right|$ has the greatest magnitude when $y = y_{\ell + 1}$. Because of this, we use the substitution $y_{\ell + 1} = \frac{k-1}{2 \pi (\ell + 1)}$ and Lemma \ref{lem:Functions} to obtain:
\begin{align*}
    N(y,k)\lvert f_{\ell+2}(z) \rvert &\leq \frac{(k-1)^k}{\Gamma(k) (\ell + 2)} \exp(-(k-1)) \left( 1 + \frac{1}{\ell + 1} \right)^k \exp\left( - \frac{k-1}{\ell + 1} \right) \\
    &= \frac{e}{\ell + 2} \left( \frac{k-1}{e} \right)^k \frac{1}{\Gamma(k)} \exp\left( -k \cdot R\left( \frac{1}{\ell + 1} \right) \right)\exp\left(\frac{1}{\ell+1}\right)\\
    &\ll\left( \frac{k-1}{e} \right)^k \frac{1}{\Gamma(k)} \exp\left( -k \cdot R\left( \frac{1}{\ell + 1} \right)\right).
\end{align*} 
Then, using \eqref{eq:Stirling}, we find that
\begin{equation}\label{eq:FunctionBounds}
    N(y,k)\lvert f_{\ell+2}(z) \rvert \ll \frac{\sqrt{k}}{\ell} \exp\left( -k \cdot R\left( \frac{1}{\ell + 1} \right) \right).
\end{equation}
For $N(y,k)\left|f_{\ell - 1}(z)\right|$, the function has the greatest magnitude when $y = y_{\ell}$ in the region $V_{\ell}$. Then, we proceed in the same fashion as before to achieve the bound in \eqref{eq:FunctionBounds} for $N(y,k)\left|f_{\ell - 1}(z)\right|$. We find that the upper bound on $N(y,k)|f_{\ell+3}(z)|$ is no larger than the upper bound on $N(y,k)|f_{\ell+2}(z)|$ for $y$ in the region \eqref{eq:YBounds}. This is similarly true for $N(y,k)|f_{\ell-2}(z)|$ and $N(y,k)|f_{\ell-1}(z)|$, respectively.

\vskip 0.1 in
\textbf{Part 2:} Let $\varepsilon_1(z)$ and $\varepsilon_2(z)$ be defined as in Section \ref{sec:MainTerm}. By [RVY, p.18],
\begin{equation*}
    N(y,k)\left| \varepsilon_2(z) \right| \ll Q(k,2\pi (\ell + 3) y),
\end{equation*}
and
\begin{equation*}
    N(y,k)\left| \varepsilon_1(z) \right| \ll P(k,2\pi (\ell - 2) y),
\end{equation*}
where $Q(s,x)$ is the normalized incomplete gamma function and $P(s,x)$ is the complementary incomplete gamma function, defined by:
\begin{equation*}
    Q(s,x) = \frac{1}{\Gamma(s)} \int_x^{\infty} t^s e^{-t} \frac{dt}{t}, \qquad P(s,x) = \frac{1}{\Gamma(s)} \int_0^x t^s e^{-t} \frac{dt}{t}.
\end{equation*}
Using the results of [T], [RVY] derived
\begin{equation}\label{eq:Temme}
    Q(s,x) \ll \exp\left( -\frac{(x-s)^2}{4s} \right) + \exp\left( -\frac{\left| x-s \right|}{4} \right) .
\end{equation}
Note that $P(s,x)$ is bounded above by \eqref{eq:Temme} as well. Letting $s = k$ and $x = 2\pi (\ell + 3)y$, and using the inequality $\frac{k-1}{\ell + 1} \leq 2\pi y \leq \frac{k-1}{\ell}$, we have
\begin{equation*}
    x - s = 2\pi (\ell + 3) y - k \geq \frac{2k}{\ell + 1} + O(1).
\end{equation*}
Thus, with \eqref{eq:Temme}, we derive
\begin{equation*}
    Q(k,2\pi (\ell + 3)y) \ll \exp\left( -\frac{k}{(\ell+1)^2} \right).
\end{equation*}
The same bound holds for $P(k, 2\pi (\ell - 2)y)$ as well. Since $R(t) < t^2$, the bounds on $N(y,k)|\varepsilon_1(z)|$ and $N(y,k)|\varepsilon_2(z)|$ are consistent with \eqref{eq:BoundonErr}.
\vskip 0.1 in
\textbf{Part 3:} Recall $\delta(z)$ satisfies \eqref{eq:delta}. Note that
\begin{equation}\label{eq:NestedSum}
    \sum_{\substack{b \mid n \\ b < n}} \left( \frac{b}{n} \right)^{k-1} \leq \sum_{d > 1} \frac{1}{d^{k-1}} = \zeta(k-1) - 1 \ll \frac{1}{2^k}.
\end{equation}
The method of proof given in \eqref{eq:UpperonH}, Part 1, and Part 2, gives us the following bound in $V_{\ell}$:
\begin{equation}\label{eq:FourierBound}
    N(y,k)\sum_{n=1}^{\infty} \left| f_n(z) \right| \ll \frac{\sqrt{k}}{\ell}.
\end{equation}
(Note that [RVY] utilizes the triangle inequality.) Thus, from \eqref{eq:FourierBound} and \eqref{eq:NestedSum}, we find that
\begin{equation*}
    N(y,k)\left| \delta(z) \right| \ll \frac{\sqrt{k}}{2^k \ell}. \qedhere
\end{equation*}
\end{proof}

\begin{proof}[Proof of Lemma \ref{lem:Functions}]
Let $0 < t \leq \frac{1}{2}$. By taking the derivative of $R(t)$, we find that $R(t)$ is strictly increasing on its domain. Using its power series expansion, we see that $R(t) \leq \frac{1}{2} t^2 < t^2$. To show that $M(t) < R(t)$, we define $w(t)$ to be the infinite sum $w(t) = \frac{1}{2}t + \frac{1}{3} t^2 + \frac{1}{4}t^3 + \cdots$. Then, $M(t) = R(w(t)) = w(t) - \log(1 + w(t))$, and our proof will be finalized as long as $w(t) < t$ by the monotonicity of $R(t)$. Note that, for $t<1$,
\begin{equation*}
    w(t) \leq \frac{1}{2}t + \frac{1}{3} \left( \frac{t^2}{1 - t} \right) < t. \qedhere
\end{equation*}

\end{proof}

\section{Inequivalence of Zeros at Atkin-Lehner Cusps}\label{AL}
In order to show inequivalence of zeros around certain cusps, we recall the theory of Atkin-Lehner involutions. We write $N=q_1 q_2=QR$ such that $(Q,R)=1$. An Atkin-Lehner involution $W_Q^N$ is given by
\begin{align*}
W_Q^N=\begin{pmatrix} Qa & b\\Nc &Qd\end{pmatrix},
\end{align*}
where $a\equiv 1\mod{R}$, $b\equiv 1\mod{Q}$ and $\det W_Q^N =Q$.\\

Let $\chi_1'$ and $\chi_2'$ be primitive characters modulo $q_1'$ and $q_2'$, respectively. We define $q_1=(q_1',R)\cdot(q_2',Q)$ and $q_2=(q_1',Q)\cdot(q_2',R)$. Factor each character uniquely as $\chi_1'=\chi_1'^{(Q)}\chi_1'^{(R)}$ and $\chi_2'=\chi_2'^{(Q)}\chi_2'^{(R)}$, where $\chi_i'^{(Q)}$ has modulus $(q_i',Q)$, and similarly for $R$. We define $\chi_1 =\chi_1'^{(R)}\chi_2'^{(Q)}$ and $\chi_2 =\chi_1'^{(Q)}\chi_2'^{(R)}$, and note that $\chi_i$ is a primitive character modulo $q_i$.\\

Weisinger showed in [W] that
\begin{align*}
E_{\chi_1',\chi_2',k}\Big|_{W_Q^N}(z)=c_Q E_{\chi_1,\chi_2,k}(z),
\end{align*}
for some constant $c_Q$. This allows one to study the behavior of $E_{\chi_1,\chi_2,k}(z)$ around the cusp at infinity in terms of the series $E_{\chi_1',\chi_2',k}(z)$ around the cusp at $\frac{a}{Rc}$. We will use this theory in the coming proposition.\\

Let $\mathcal{K}(q_2)$ denote the region 
\begin{align*}
\mathcal{K}(q_2)=\left\{z=x+iy\in\mathbb{H}: y>\frac{1}{q_2},\ -\frac{1}{2}<x\leq\frac{1}{2}\right\}.
\end{align*}
\begin{prop} For an Atkin-Lehner involution $W_Q^N$, where $Q\neq 1$, the image of $\mathcal{K}(q_2')$ under the involution $W_Q^N$ is disjoint with the region $\mathcal{K}(q_2)$.
\end{prop}
\begin{proof}
We briefly state that we cannot have $c=0$, since that would imply $Q=1$. Let $z=x+iy\in\mathcal{K}(q_2')$. Then
\begin{align*}
\Imag \left(W_Q^N z\right)&=\frac{Qy}{(Ncx+Qd)^2+(Ncy)^2}\leq\frac{Qy}{N^2 c^2 y^2}\\
&=\frac{1}{Rc^2 Ny} <\frac{q_2'}{Rc^2 N}\leq\frac{q_2'}{RN}.
\end{align*}
Note that
\begin{align*}
\frac{q_2'}{RN}\leq\frac{1}{q_2},
\end{align*}

since
\begin{align*}
\frac{q_2'q_2}{RN}=\frac{q_2}{q_1'R}=\frac{(q_1',Q)}{q_1'}\cdot\frac{(q_2',R)}{R}\leq 1.
\end{align*}
Therefore $\Imag(W_Q^N z)\not\in\mathcal{K}(q_2)$.
\end{proof}

\begin{cor} The zeros of $E_{\chi_1,\chi_2,k}(z)$ in $\mathcal{K}(q_2)$ are $\Gamma_0(q_1q_2)$-inequivalent to the image of the zeros of $E_{\chi_1',\chi_2',k}(z)$ under the Atkin-Lehner involution $W_Q^N$.
\end{cor}

\begin{section}*{Acknowledgements}
The authors would like to thank Dr. Matthew Young for his immense support and guidance. Additionally, the authors are grateful to Nathan Green for his advice, Texas A\&M University Department of Mathematics, and the National Science Foundation for funding the Texas A\&M REU, grant DMS--1460766.
\end{section}

\end{document}